\documentclass[11pt, twoside, leqno]{article}

\usepackage{amssymb}
\usepackage{amsmath}
\usepackage{amsthm}
\usepackage{color}
\usepackage{mathrsfs}

\allowdisplaybreaks

\def\rr{{\mathbb R}}
\def\rn{{{\rr}^n}}
\def\zz{{\mathbb Z}}
\def\cc{{\mathbb C}}
\def\nn{{\mathbb N}}

\def\ca{{\mathcal A}}

\def\cf{{\mathcal F}}

\def\cl{{\mathcal L}}

\def\cp{{\mathcal P}}

\def\cs{{\mathcal S}}

\def\mi{{\mathrm I}}

\def\fz{\infty}
\def\az{\alpha}
\def\bz{\beta}
\def\dz{\delta}

\def\ez{\epsilon}
\def\gz{{\gamma}}
\def\bgz{{\Gamma}}

\def\lz{\lambda}
\def\blz{\Lambda}

\def\tz{\theta}
\def\sz{\sigma}
\def\vz{\varphi}
\def\vz{\varphi}

\def\lf{\left}
\def\r{\right}
\def\lfz{{\lfloor}}
\def\rfz{{\rfloor}}

\def\hs{\hspace{0.25cm}}
\def\ls{\lesssim}

\def\noz{\nonumber}
\def\wz{\widetilde}
\def\wh{\widehat}
\def\st{\subset}

\def\bh{\backslash}

\def\dxt{\,\frac{dx\,dt}{t}}
\def\dxtn{\,\frac{dx\,dt}{t^{n+1}}}

\def\dtn{\,\frac{dt}{t^{n+1}}}

\def\dis{\displaystyle}

\def\supp{\mathop\mathrm{\,supp\,}}

\def\loc{{\mathop\mathrm{\,loc\,}}}
\def\essinf{\mathop\mathrm{\,essinf\,}}
\def\esup{\mathop\mathrm{\,esssup\,}}
\def\fin{{\mathop\mathrm{fin}}}

\def\aa{{\mathbb A}}
\def\lv{{L^{\vz}(\rn)}}
\def\hv{{H^{\vz}(\rn)}}
\def\hp{{H^{p}(\rn)}}
\def\vb{{\|\chi_B\|_\lv^{-1}}}

\def\vbl{{\|\chi_{B_0}\|_\lv^{-1}}}

\def\hvf{{H^{\vz,q,s}_{\rm fin}(\rn)}}

\def\hva{{H^{\vz,q,s}_{\rm at}(\rn)}}

\def\lvb{{L_{\vz,s}^q(B)}}
\def\lvq{{L_{\vz}^q(B)}}
\def\bmo{{\mathop\mathrm {BMO}}}
\def\bmov{{{\mathop\mathrm {BMO}}^\vz(\rn)}}

\def\lqs{{\cl_{\vz,q,s}(\rn)}}
\def\lqps{{\cl_{\vz,q',s}(\rn)}}

\def\lys{{\cl_{\vz,1,s}(\rn)}}
\def\lym{{\cl_{\vz,1,m(\vz)}(\rn)}}

\def\ps{{P_{B}^s}}
\def\pse{{P_{2B}^s}}
\def\psl{{P_{B_0}^s}}
\def\psel{{P_{2B_0}^s}}

\newtheorem{thm}{Theorem}[section]

\newtheorem{lem}[thm]{Lemma}
\newtheorem{cor}[thm]{Corollary}
\theoremstyle{definition}
\newtheorem{defn}[thm]{Definition}
\newtheorem{rem}[thm]{Remark}

\numberwithin{equation}{section}

\begin{document}

\arraycolsep=1pt

\title{\bf\Large Musielak-Orlicz Campanato Spaces and Applications
\footnotetext{\hspace{-0.35cm} 2010 {\it
Mathematics Subject Classification}. Primary 42B25; Secondary 42B30, 42B35, 46E30.
\endgraf {\it Key words and phrases}. Musielak-Orlicz function, BMO space,
Campanato space, John-Nirenberg inequality, dual space, Carleson measure.
\endgraf
Dachun Yang is supported by the National
Natural Science Foundation  of China (Grant No. 11171027) and
the Specialized Research Fund for the Doctoral Program of Higher Education
of China (Grand No. 20120003110003).} }
\author{Yiyu Liang and Dachun Yang\,\footnote{Corresponding author}}
\date{}

\maketitle


\vspace{-0.6cm}

\begin{center}
\begin{minipage}{13cm}
{\small {\bf Abstract}\quad
Let $\varphi: \mathbb R^n\times [0,\infty)\to[0,\infty)$
be such that $\varphi(x,\cdot)$ is an Orlicz function
and $\varphi(\cdot,t)$ is
a Muckenhoupt $A_\infty(\mathbb R^n)$ weight uniformly in $t$.
In this article, the authors introduce
the Musielak-Orlicz Campanato space ${\mathcal L}_{\varphi,q,s}({\mathbb R}^n)$
and, as an application, prove that some of them is the dual space of the Musielak-Orlicz
Hardy space $H^{\varphi}(\mathbb R^n)$, which in the case when $q=1$ and $s=0$
was obtained by L. D. Ky [arXiv: 1105.0486].
The authors also establish a John-Nirenberg inequality for functions
in ${\mathcal L}_{\varphi,1,s}({\mathbb R}^n)$ and, as an application,
the authors also obtain several equivalent characterizations
of ${\mathcal L}_{\varphi,q,s}({\mathbb R}^n)$, which, in return, further induce
the $\varphi$-Carleson measure characterization of
${\mathcal L}_{\varphi,1,s}({\mathbb R}^n)$.
}
\end{minipage}
\end{center}


\vspace{0.4543pt}

\section{Introduction\label{s1}}

\hskip\parindent The \emph{$\mathrm{BMO}$ space} $\bmo(\rn)$,
originally introduced by John and Nirenberg \cite{jn},
is defined as the space of all locally integrable functions $f$ satisfying
$$\|f\|_{\mathrm{BMO}(\mathbb R^n)}
:= \sup_{B\subset\rn}\frac{1}{|B|}\int_{B}|f(x)-f_B|\,dx<\infty,$$
where the supremum is taken over all balls $B\subset \mathbb R^n$ and
$f_B:=\frac{1}{|B|}\int_Bf(x)\,dx$.
Fefferman and Stein \cite{fs82} proved that
$\bmo$ is the dual space of the Hardy space $H^1(\rn)$. The space
$\mathrm{BMO}(\mathbb R^n)$ is also considered as a natural substitute for
$L^\infty(\mathbb{R}^n)$ when studying the boundedness of operators.

For any $s\in\zz_+:=\{0,1,\dots\}$,
let $\cp_s(\rn)$ denote the \emph{polynomials with order not more than $s$}.
Assume that $f$ is a locally integrable function on $\rn$.
For any ball $B\st\rn$ and $s\in\zz_+$,
let $P_B^sg$ be the \emph{unique polynomial} $P\in\cp_s(\rn)$
such that, for all $Q\in\cp_s(\rn)$,
$$\int_B[g(x)-P(x)]Q(x)\,dx=0.$$
Recall that, for $\bz\in[0,\fz)$, $s\in\zz_+$ and $q\in[0,\fz)$,
a locally integrable function $f$ is said to belong to
the \emph{Campanato spaces $L_{\bz,q,s}(\rn)$}
introduced by Campanato \cite{c64}, if
\begin{eqnarray}\label{cam}
\|f\|_{L_{\bz,q,s}(\mathbb R^n)}
:= \sup_{B\subset\rn}|B|^{-\bz}
\lf\{\frac{1}{|B|}\int_{B}|f(x)-P_B^sf(x)|^q\,dx\r\}^{1/q}<\infty,
\end{eqnarray}
where the supremum is taken over all balls $B$ in $\mathbb{R}^n$.

Obviously, $L_{0,1,0}(\rn)=\bmo(\rn)$ since $P_B^0f=f_B$.
Moreover, Taibelson and Weiss \cite{tw80} further showed that,
for all $q\in[1,\fz)$ and $s\in\zz_+$,
$L_{0,q,s}(\rn)$ and $\bmo(\rn)$ coincide with equivalent norms.
Taibelson and Weiss \cite{tw80} also proved that
the dual space of the Hardy space $\hp$ with $p\in(0,1]$ is
the space $L_{\frac1p-1,q,s}(\rn)$ for $q\in[1,\fz)$ and
$s\ge\lfz n(\frac1p-1)\rfz$. Here and in what follows, we use
the \emph{symbol $\lfz s\rfz$} for any $s\in\rr$ to denote
the maximal integer not more than $s$. For more applications of
Campanato spaces and those function spaces related
to Campanato spaces in harmonic analysis and partial differential equations,
see, for example, \cite{p69, tw80, ae02, g02, dxy07, n07,
hmy07, n10, ysy, ft11} and their references.

On the other hand, as a generalization of $L^p(\rn)$,
the Orlicz space was introduced by
Birnbaum-Orlicz \cite{bo31} and Orlicz \cite{o32}.
Recently, Ky \cite{ky} introduced a new \emph{Musielak-Orlicz Hardy space $\hv$},
which generalizes both the Orlicz-Hardy space (see, for example, \cite{j80,v87})
and the weighted Hardy space (see,
for example, \cite{g79, gr85, st89}).
Musielak-Orlicz functions are the natural generalization of Orlicz functions
that may vary in the spatial variables; see, for example, \cite{m83}.
The motivation to study function spaces of Musielak-Orlicz type comes
from applications to elasticity, fluid dynamics, image processing,
nonlinear partial differential equations and the calculus of variation;
see, for example, \cite{bg10,bgk12,bijz07,d05,dhr09,ky} and their references.
It is also worth noticing that
some special Musielak-Orlicz Hardy spaces appear naturally
in the study of the products of functions in $\bmo(\rn)$ and
$H^1(\rn)$ (see \cite{bgk12,bijz07}), and the endpoint estimates for
the div-curl lemma and the commutators of singular integral operators
(see \cite{bfg10,bgk12,ky2}).

Ky \cite{ky} also introduced the Musielak-Orlicz BMO-type space $\bmov$,
which generalizes the classical space $\mathrm{BMO}(\mathbb{R}^n)$,
the weighted BMO
space $\bmo_w(\rn)$ (see, for example, \cite{mw76}) and the Orlicz BMO-type space
$\bmo_\rho(\rn)$ (see, for example, \cite{s79, j80, v87}).
Ky \cite{ky} proved that the dual space of $\hv$
is the Musielak-Orlicz BMO space $\bmov$ \emph{in the case when $m(\vz)=0$},
where $m(\vz):=\lfz n(\frac{q(\vz)}{i(\vz)}-1)\rfz$,
$i(\vz)$ and $q(\vz)$ are the critical uniformly lower type index
and the critical weight index of $\vz$, respectively;
see \eqref{ivz} and \eqref{qvz} below.
Recall that a locally integrable function $f$ on $\rn$ is said to belong to
the \emph{space} $\mathrm{BMO}^{\varphi}(\mathbb{R}^n)$, if
$$\|f\|_{\mathrm{BMO}^{\varphi}(\mathbb{R}^n)}:=
\sup_{B\subset\rn}\frac{1}{\|\chi_B\|_{L^{\varphi}
(\mathbb{R}^n)}}\int_{B}|f(x)-f_{B}|\,dx<\infty,$$
where the supremum is taken over all balls $B$ in $\mathbb{R}^n$,
$\chi_B$ denotes the \emph{characteristic function} of $B$, and
$$\|\chi_B\|_{L^{\varphi}(\mathbb{R}^n)}:=
\inf\left\{\lambda\in(0,\infty):\
\int_{B}\varphi
\left(x, \frac {1}{\lambda}\right)
\,dx\leq1\right\}.$$
As an application, Ky \cite{ky} proved that the class of pointwise multipliers
for $\mathrm{BMO}(\mathbb{R}^n)$ characterized by Nakai and Yabuta \cite{ny85} is just
the space $L^{\infty}(\mathbb{R}^n)\cap
\mathrm{BMO}^{\mathrm{log}}(\mathbb{R}^n)$ (see \cite{ky}),
where $\mathrm{BMO}^{\mathrm{log}}(\mathbb{R}^n)$ denotes the \emph{Musielak-Orlicz BMO-type space}
related to the growth function
$$\varphi(x,t):=\frac{t}{\ln(e+|x|)+\ln(e+t)}$$
for all $x\in\mathbb{R}^n$ and $t\in[0,\infty)$.

To complete the study of Ky \cite{ky} on the dual space of
$\hv$, namely, to decide the dual space of Hardy space $\hv$
\emph{in the case when $m(\vz)\in\nn$}, we need to introduce
the following Musielak-Orlicz Campanato spaces.

\begin{defn}\label{d-cam}
Let $\vz$ be as in Definition \ref{d-vz}, $q\in[1,\fz)$ and $s\in\zz_+$.
A locally integrable function $f$ on $\rn$
is said to belong to the
\emph{Musielak-Orlicz Campanato space $\cl_{\vz,q,s}(\rn)$}, if
\begin{eqnarray*}
\|f\|_{\lqs}&:=&\sup_{B\subset\rn}\!\frac{1}{\|\chi_{B}
\|_{L^{\vz}(\rn)}}
\lf\{\!\!\int_{B}\!\!
\lf[\frac{\lf|f(x)-P_B^sf(x)\r|}{\vz(x,\vb)}\r]^q
\!\!\vz\!\lf(x,\vb\r)\!dx\r\}^{1/q}
\!\!\!<\fz,
\end{eqnarray*}
where the supremum is taken over all the balls
$B\subset\rn$.
\end{defn}

As usual, by abuse of notation, we identify $f\in\lqs$ with
$f+\cp_s(\rn)$.

\begin{rem}\label{r-def}
(i) When $\vz(x,t):=t^p$, with $p\in(0,1]$,
for all $x\in\rn$ and $t\in(0,\fz)$,
by some computations, we know that $\|\chi_B\|_\lv=|B|^{1/p}$ and $\vz(x,\vb)=|B|^{-1}$
for any ball $B\st\rn$ and $x\in\rn$. Thus, in this case,
$\lqs$ is just the classical Campanato space
$L_{\frac1p-1,q,s}(\rn)$ (see \eqref{cam}),
which was introduced by Campanato \cite{c64}.

(ii) When $\vz(x,t):=w(x)t^p$, with $p\in(0,1]$
and $w\in A_{\fz}(\rn)$, for all $x\in\rn$ and $t\in(0,\fz)$,
via some computations, we see that
$$\|\chi_B\|_\lv=[w(B)]^{1/p}\quad \mathrm{and}\quad \vz\lf(x,\vb\r)=[w(B)]^{-1}$$
for any ball $B\st\rn$ and $x\in\rn$, where
$w(B):=\int_{B}w(x)\,dx$. Thus, in this case, the space $\lqs$
coincides with the weighted Campanato space
introduced by Garc\'ia-Cuerva \cite{g79}
as the dual space of the corresponding weighted Hardy spaces.
\end{rem}

This article is organized as follows.

In Section \ref{s2}, we recall some
notions concerning growth functions and some
of their properties established in \cite{ky}.
Then via some skillful applications of these properties on growth functions
and some estimate of the minimal polynomial from Taibleson and Weiss \cite{tw80},
we establish a John-Nirenberg inequality for functions in $\lys$; see Theorem \ref{J-N} below.
To obtain this, we need to overcome some essential difficulties caused by the inseparability of
the space variant $x$ and the time variant $t$ appeared in $\vz(x,t)$.
A new idea for this is to choose $t=\vb$, which brings us some convenient estimates such as,
for all balls $B$, $\vz(B,\vb)=1$ and there exists a positive constant $C$ such that, for all balls $\wz B\st B$,  
$$\frac{\vz(B,\vb)}{\|\chi_B\|_\lv}
\le C\frac{\vz(\wz B,\vb)}{\|\chi_{\wz B}\|_\lv}.$$
As an application of the John-Nirenberg inequality,
in Theorem \ref{char} below,
we further prove that $\lys=\lqs$ with $q\in[1,q(\vz)')$
and some other equivalent characterizations for $\lqs$, where here and
in what follows, $r'$ denotes the \emph{conjugate index} of $r\in[1,\fz]$.
Even when $\vz$ is as in Remark \ref{r-def}(ii) with $p\in (0,1)$,
Theorems \ref{J-N} and \ref{char} are also new; see Remarks \ref{r-JN} and
\ref{r-char} below.

In Section \ref{s3}, applying the equivalent characterizations of
$\lqs$ in Section \ref{s2}, we prove that the dual space of $\hv$
is the space $\lys$ for all $s\in [m(\vz),\fz)\cap\zz_+$ and $m(\vz)\in\zz_+$,
which further completes the dual result of Ky \cite{ky} in the case $m(\vz)=0$;
see Theorem \ref{dual} below. As a corollary, we further conclude
that $\lqs$ and $\cl_{\vz, 1, m(\vz)}(\rn)$ coincide with equivalent norms
for all $q\in[1,q(\vz)')$ and $s\in [m(\vz),\fz)\cap\zz_+$;
see Corollary \ref{c3.1} below.

Section \ref{s4} is devoted to establish a $\vz$-Carleson measure
characterization of $\lys$; see Theorem \ref{carl} below.
To this end, we need to use the Lusin area function characterization of $\hv$ established in \cite{hyy} and
the equivalent characterizations obtained in Theorem \ref{char}.
Even when $\vz$ is as in Remark \ref{r-def}(ii) with $p\in(0,1)$ and
$w\in A_1(\rn)$, Theorem \ref{carl} is also new;
see Remark \ref{r-carl} below.

Except to give out the dual space of $\hv$ in the case when $m(\vz)\in\nn$,
another interesting application of the Musielak-Orlicz Campanato spaces
$\lqs$ exists in establishing the intrinsic Littlewood-Paley function
characterizations of the Hardy space $\hv$, which will be given in \cite{ly}.
The dual space $(\lym)^*$ of the Musielak-Orlicz Campanato space
$\lym$, together with the fact that $\lym$ is the dual space
of $\hv$, will play a key role in \cite{ly}.

Finally we make some conventions on notation. Throughout the whole
paper, we denote by $C$ a \emph{positive constant} which is
independent of the main parameters, but it may vary from line to
line. The {\it symbol} $A\ls B$ means that $A\le CB$. If $A\ls
B$ and $B\ls A$, then we write $A\sim B$.
For any measurable subset $E$ of $\rn$, we denote by $E^\complement$ the {\it set}
$\rn\setminus E$ and its \emph{characteristic function} by $\chi_{E}$.
We also set $\nn:=\{1,\,2,\,
\ldots\}$ and $\zz_+:=\nn\cup\{0\}$.

\section{The John-Nirenberg Inequality and Equivalent Characterizations \label{s2}}

\hskip\parindent In this section, we prove a John-Nirenberg inequality
for functions in $\lys$,
by which we further establish some equivalent characterizations for $\lqs$.

Recall that a function
$\Phi:[0,\fz)\to[0,\fz)$ is called an \emph{Orlicz function} if it
is nondecreasing, $\Phi(0)=0$, $\Phi(t)>0$ for all $t\in(0,\fz)$ and
$\lim_{t\to\fz}\Phi(t)=\fz$.
The function $\Phi$ is said to be of
\emph{upper type $p$} (resp. \emph{lower type $p$}) for some $p\in[0,\fz)$, if
there exists a positive constant $C$ such that, for all
$t\in[1,\fz)$ (resp. $t\in[0,1]$) and $s\in[0,\fz)$,
$\Phi(st)\le Ct^p \Phi(s).$

For a given function $\vz:\,\rn\times[0,\fz)\to[0,\fz)$ such that, for
any $x\in\rn$, $\vz(x,\cdot)$ is an Orlicz function,
$\vz$ is said to be of \emph{uniformly upper type $p$} (resp.
\emph{uniformly lower type $p$}) for some $p\in[0,\fz)$ if there
exists a positive constant $C$ such that, for all $x\in\rn$,
$t\in[0,\fz)$ and $s\in[1,\fz)$ (resp. $s\in[0,1]$), $\vz(x,st)\le Cs^p\vz(x,t)$.
We say that $\vz$ is of \emph{positive uniformly upper type}
(resp. \emph{uniformly lower type}) if it is of uniformly upper
type (resp. uniformly lower type) $p$ for some $p\in(0,\fz)$.
The \emph{critical uniformly lower type index} of $\vz$ is defined by
\begin{equation}\label{ivz}
i(\vz):=\sup\{p\in(0,\fz):\ \vz\ \text{is of uniformly lower
type}\ p\}.
\end{equation}
Observe that $i(\vz)$ may not be attainable, namely, $\vz$ may
not be of uniformly lower type $i(\vz)$ (see \cite{lhy}).

Let $\vz:\rn\times[0,\fz)\to[0,\fz)$ satisfy that
$x\mapsto\vz(x,t)$ is measurable for all $t\in[0,\fz)$. Following
\cite{ky}, $\vz(\cdot,t)$ is said to be \emph{uniformly locally
integrable} if, for all compact sets $K$ in $\rn$,
$$\int_{K}\sup_{t\in(0,\fz)}\lf\{|\vz(x,t)|
\lf[\int_{K}|\vz(y,t)|\,dy\r]^{-1}\r\}\,dx<\fz.$$

The function $\vz(\cdot,t)$ is said to satisfy the
\emph{uniformly Muckenhoupt condition for some $q\in[1,\fz)$},
denoted by $\vz\in\aa_q(\rn)$, if
$\vz$ is uniformly locally integrable and, when $q\in (1,\fz)$,
\begin{equation*}
\sup_{t\in
(0,\fz)}\sup_{B\subset\rn}\frac{1}{|B|^q}\int_B
\vz(x,t)\,dx \lf\{\int_B
[\vz(y,t)]^{-q'/q}\,dy\r\}^{q/q'}<\fz,
\end{equation*}
where $1/q+1/q'=1$, or, when $q=1$,
\begin{equation*}
\sup_{t\in (0,\fz)}
\sup_{B\subset\rn}\frac{1}{|B|}\int_B \vz(x,t)\,dx
\lf(\esup_{y\in B}[\vz(y,t)]^{-1}\r)<\fz.
\end{equation*}
Here the first supremums are taken over all $t\in[0,\fz)$ and the
second ones over all balls $B\subset\rn$.

Let $\aa_{\fz}(\rn):=\cup_{q\in[1,\fz)}\aa_{q}(\rn)$.
The \emph{critical weight index} of $\vz\in\aa_{\fz}(\rn)$ is defined as follows:
\begin{equation}\label{qvz}
q(\vz):=\inf\lf\{q\in[1,\fz):\ \vz\in\aa_{q}(\rn)\r\}.
\end{equation}
Now we recall the notion of growth functions (see \cite{ky}).

\begin{defn}\label{d-vz}
A function $ \varphi:\mathbb R^n\times[0,\infty)\to[0,\infty)$
is called a \emph{growth function} if
the following conditions are satisfied:
\vspace{-0.25cm}
\begin{enumerate}
\item[(i)] $\vz$ is a \emph{Musielak-Orlicz function}, namely,
\vspace{-0.2cm}
\begin{enumerate}
    \item[(i)$_1$] the function $\vz(x,\cdot):\ [0,\fz)\to[0,\fz)$ is an
    Orlicz function for all $x\in\rn$;
    \vspace{-0.2cm}
    \item [(i)$_2$] the function $\vz(\cdot,t)$ is a measurable
    function for all $t\in[0,\fz)$.
\end{enumerate}
\vspace{-0.25cm} \item[(ii)] $\vz\in \aa_{\fz}(\rn)$.
\vspace{-0.25cm} \item[(iii)] $\vz$ is of positive
uniformly lower type $p$ for some $p\in(0,1]$ and of uniformly
upper type 1.
\end{enumerate}
\end{defn}

Throughout the whole paper, we \emph{always
assume that $\vz$ is a growth function} as in Definition
\ref{d-vz} and, for any measurable subset $E$ of $\rn$ and $t\in[0,\fz)$,
we \emph{denote $\int_E\vz(x,t)\,dx$ by $\vz(E,t)$}.
Let us now introduce the Musielak-Orlicz space.

The \emph{Musielak-Orlicz space $L^{\vz}(\rn)$} is defined to be the space
of all measurable functions $f$ such that
$\int_{\rn}\vz(x,|f(x)|)\,dx<\fz$ with the \emph{Luxembourg norm}
$$\|f\|_{L^{\vz}(\rn)}:=\inf\lf\{\lz\in(0,\fz):\ \int_{\rn}
\vz\lf(x,\frac{|f(x)|}{\lz}\r)\,dx\le1\r\}.$$

To establish a John-Nirengerg inequality for functions in $\lys$,
we need the following lemmas.
Observe that Lemmas \ref{C-Z} and \ref{weight}
are just \cite[Lemmas 3.2 and 3.1]{mw76}.

\begin{lem}\label{C-Z}
Let $w$ be a measure satisfying the doubling condition,
namely, there exists a positive constant $C_0$
such that, for all balls $B\st\rn$, $w(2B)\le C_0w(B)$
and, for a given ball $B\st\rn$ and $\sz$,
let $f$ be a nonnegative function which satisfies that
$$\frac1{w(B)}\int_Bf(x)w(x)\,dx\le \sz.$$
Then there exist non-overlapping balls $\{B_k\}_{k\in\nn}$
and a positive constant $\wz C$, depending only on $C_0$,
such that
$f(x)\le \sz$ for almost every $x\in B\bh\cup_{k\in\nn}B_k$ and
$$\sz\le\frac 1{w(B_k)}\int_{B_k}fw\,dx
\le \wz C\sz \mbox{ \ for all \ }k\in\nn. $$
\end{lem}

\begin{lem}\label{weight}
Let $q\in(1,\fz)$ and $1/q+1/q'=1$.
If $w\in A_q(\rn)$, then there exists a positive constant $C$
such that, for all balls $B\st\rn$ and $\bz\in(0,\fz)$,
$$w(\{x\in B:\ w(x)<\bz\})\le C\lf[\bz\frac{|B|}{w(B)}\r]^{q'}w(B).$$
\end{lem}

The following Lemma \ref{PBg} is from \cite[p. 83]{tw80}.

\begin{lem}\label{PBg}
Let $g\in L_\loc^1(\rn)$, $s\in\zz_+$ and $B$ be a ball in $\rn$.
Then there exists a positive constants $C$,
independent of $g$ and $B$, such that
$$\sup_{x\in B}|P_B^sg(x)|\le\frac C{|B|}\int_B|g(x)|\,dx.$$
\end{lem}

Now, we can state the John-Nirenberg inequality for functions
in $\lys$ as follows.

\begin{thm}\label{J-N}
Let $\vz$ be as in Definition \ref{d-vz} and $f\in\lys$.
Then there exist positive constants $C_1$, $C_2$ and $C_3$, independent of $f$,
such that, for all balls $B\subset\rn$ and $\az\in(0,\fz)$,
when $\vz\in\aa_1(\rn)$,
\begin{eqnarray*}
&&\vz\lf(\lf\{x\in B:\ \frac{|f(x)-P_B^sf(x)|}{\vz(x,\vb)}>\az\r\},
\vb \r)\\
&&\hs\le C_1\exp\lf\{-\frac{C_2\az}
{\|f\|_{\lys}\|\chi_B\|_\lv}\r\}
\end{eqnarray*}
and, when $\vz\in\aa_q(\rn)$ for some $q\in(1,\fz)$,
\begin{eqnarray*}
&&\vz\lf(\lf\{x\in B:\ \frac{|f(x)-P_B^sf(x)|}{\vz(x,\vb)}>\az\r\},
\vb \r)\\
&&\hs\le C_3\lf[1+\frac{\az}
{\|f\|_{\lys}\|\chi_B\|_\lv}\r]^{-q'},
\end{eqnarray*}
where $1/q+1/q'=1$.
\end{thm}

\begin{proof}
Let $f\in\lys$. Fix any ball $B_0\subset\rn$.
Without loss of generality, we may
assume that
$\|f\|_{\lys}=\vbl$; otherwise, we replace $f$ by
$\frac{f}{\|f\|_{\lys}\|\chi_{B_0}\|_{L^\vz(\rn)}}$.
For any $\az\in(0,\fz)$ and ball $B\st B_0$,
let
$$\lz(\az,B)
:=\vz\lf(\lf\{x\in B:\ \frac{|f(x)-P_B^sf(x)|}{\vz(x,\vbl)}>\az\r\},\vbl \r)$$
and
\begin{eqnarray}\label{jn1}
\cf(\az):=\sup_{B\st B_0}\frac{\lz(\az,B)}{\vz(B,\vbl)}.
\end{eqnarray}
By $\lz(\az,B)\le\vz(B,\vbl)$, we see that, for all $\az\in(0,\fz)$,
$\cf(\az)\le1$.

From the upper type $1$ property of $\vz$,
$\|f\|_{\lys}=\vbl$ and 
$$\vz\lf(B,\vb\r)=1,$$
it follows that there exists a positive constant $\wz C_0$
such that, for any ball $B\st B_0$,
\begin{eqnarray}\label{jn0}
&&\frac1{\vz(B,\vbl)}\int_{B}|f(x)-P_{B}^sf(x)|\,dx\\
&&\hs\le\frac{\|\chi_{B}\|_\lv}{\vz(B,\vbl)\|\chi_{B_0}\|_\lv}\noz\\
&&\hs\le\frac{\wz C_0\|\chi_{B}\|_\lv}{\vz(B,\vb)\frac{\vbl}{\vb}
\|\chi_{B_0}\|_\lv}
=\wz C_0.\noz
\end{eqnarray}
Applying Lemma \ref{C-Z} to $B$,
$\vz(\cdot,\vbl)|f-P_{B}^sf|$ and $\sz\in[\wz C_0,\fz)$,
we know that there exist non-overlapping balls $\{B_k\}_{k\in\nn}$ in $B$
and a positive constant $\wz C_1$ as in Lemma \ref{C-Z}
such that
\begin{eqnarray}\label{cz1}
\frac{|f(x)-P_{B}^sf(x)|}{\vz(x,\vbl)}\le \sz
\quad \mbox{ for almost every } x\in B\bh[\cup_k B_k]
\end{eqnarray}
and
\begin{eqnarray}\label{cz2}
\sz\le\frac1{\vz(B_k,\vbl)}
\int_{B_k}{|f(x)-P_{B}^sf(x)|}\,dx
\le \wz C_1\sz
\mbox{ \ for all \ }k\in\nn,
\end{eqnarray}
which, together with \eqref{jn0}, implies that
\begin{eqnarray}\label{cz3}
\qquad\sum_{k=1}^\fz\vz(B_k,\vbl)
\le\frac1\sz
\int_{B}{|f(x)-P_{B}^sf(x)|}\,dx
\le \frac{\wz C_0}\sz\vz\lf(B,\vbl\r).
\end{eqnarray}
If $\sz\le\az$, \eqref{cz1} implies that, for almost every
$x\in B\bh[\cup_k B_k]$,
$\frac{|f(x)-P_{B}^sf(x)|}{\vz(x,\vbl)}\le \az$
and hence
\begin{eqnarray*}
\lz(\az,B)&&=
\vz\lf(\lf\{x\in B:\ \frac{|f(x)-P_{B}^sf(x)|}{\vz(x,\vbl)}>\az\r\},\vbl \r)\\
&&\le
\sum_{k=1}^\fz\vz\lf(\lf\{x\in B_k:\
\frac{|f(x)-P_{B}^sf(x)|}{\vz(x,\vbl)}>\az\r\},\vbl \r).\noz
\end{eqnarray*}
Thus, for $\wz C_0\le \sz\le\az$ and $0\le\gz\le\az$, it holds that
\begin{eqnarray}\label{jn2}
\lz(\az,B)
&&\le
\sum_{k=1}^\fz\lz(\az-\gz,B_k)\\
&&\hs+
\sum_{k=1}^\fz\vz\lf(\lf\{x\in B_k:\
\frac{|P_{B_k}^sf(x)-P_{B}^sf(x)|}{\vz(x,\vbl)}>\gz\r\},\vbl \r)\noz\\
&&=:\mi_1+\mi_2.\noz
\end{eqnarray}

By \eqref{jn1} and \eqref{cz3}, we have
\begin{eqnarray}\label{jn3}
\mi_1=
\sum_{k=1}^\fz\lz(\az-\gz,B_k)
&&\le
\sum_{k=1}^\fz\cf(\az-\gz)\vz(B_k,\vbl)\\
&&\le\frac{\wz C_0}\sz\cf(\az-\gz)\vz\lf(B,\vbl\r).\noz
\end{eqnarray}
On the other hand, by Lemma \ref{PBg} and \eqref{cz2},
we conclude that there exists a positive constant $\wz C_2$
as in Lemma \ref{PBg} such that, for all $x\in B_k$,
\begin{eqnarray}\label{jn4}
|P_{B_k}^sf(x)-P_{B}^sf(x)|
&&=|P_{B_k}^s(f-P_{B}^sf)(x)|
\le \frac {\wz C_2}{|B_k|}\int_{B_k}|f(x)-\ps f(x)|\,dx\\
&&\le \frac {\wz C_2\wz C_1\sz\vz(B_k,\vbl)}{|B_k|}.\noz
\end{eqnarray}

If $\vz\in\aa_1(\rn)$, then there exists a positive constant $\wz C_3$
such that
$$\frac{\vz(B_k,\vbl)}{|B_k|}\le \wz C_3\essinf_{x\in B_k}\vz(x,\vbl),$$
which, combining \eqref{jn4}, further implies that
\begin{eqnarray}\label{jn5}
&&\vz\lf(\lf\{x\in B_k:\
\frac{|P_{B_k}^sf(x)-P_{B}^sf(x)|}{\vz(x,\vbl)}>\gz\r\},\vbl \r)\\
&&\hs\le\vz\lf(\lf\{x\in B_k:\
\frac{\wz C_1\wz C_2\wz C_3\sz
\dis\essinf_{x\in B_k}\vz(x,\vbl)}{\vz(x,\vbl)}>\gz\r\},\vbl \r)\noz.
\end{eqnarray}

Now choose $\sz:=2\wz C_0$ and $\gz:=2\wz C_0\wz C_1\wz C_2\wz C_3$.
Then if $\az>\gz$, we have $\wz C_0<\sz<\az$
and $0<\gz<\az$ as required.
From \eqref{jn2} and \eqref{jn5}, it follows that
$$\mi_2
\le\sum_{k=1}^\fz\vz\lf(\lf\{x\in B_k:\
\frac{
\dis\essinf_{x\in B_k}\vz(x,\vbl)}{\vz(x,\vbl)}>1\r\},\vbl \r)=0,$$
which, combining \eqref{jn2} and \eqref{jn3}, implies that
$\lz(\az, B)\le \frac12\cf(\az-\gz)\vz(B,\vbl)$
for all $\az>\gz$ and $B\st B_0$.
Hence, $\cf(\az)\le \frac12\cf(\az-\gz)$ if $\az>\gz$.
If $m\in\nn$ and $\az$ satisfies
$m\gz<\az\le (m+1)\gz$, then
$\cf(\az)\le 2^{-1}\cf(\az-\gz)\le\cdots\le2^{-m}\cf(\az-m\gz)$.
Since $\cf(\az-m\gz)\le1$ and $m\ge\az/\gz-1$ for such $\az$,
it follows that
$$\cf(\az)\le 2^{-m}\le 2^{1-\az/\gz}=2e^{-(\frac1\gz\log2)\az}.$$
Therefore, with $C_1:=2$ and $C_2:=\frac1\gz\log2$,
for $\vz\in\aa_1(\rn)$ and $\az>\gz$,
we conclude that
\begin{eqnarray*}
\vz\lf(\lf\{x\in B_0: \frac{|f(x)-\psl f(x)|}{\vz(x,\vbl)}>\az\r\},
\vbl \r)
\le C_1 e^{-C_2\az}.
\end{eqnarray*}
This finishes the proof of Theorem \ref{J-N} in the case $\vz\in\aa_1(\rn)$.

Next, suppose $\vz\in\aa_q(\rn)$ for some $q\in(1,\fz)$.
From \eqref{cz3}, \eqref{jn2}, \eqref{jn4} and Lemma \ref{weight},
we deduce that
\begin{eqnarray*}
\mi_2&&\le\sum_{k\in\nn}\vz\lf(\lf\{x\in B_k:\
\frac {\wz C_2\wz C_1\sz\vz(B_k,\vbl)}{|B_k|\vz(x,\vbl)}>\gz\r\},\vbl \r)\noz\\
&&\le\sum_{k\in\nn}\wz C_3
\lf(\frac{\wz C_2\wz C_1\sz}\gz\r)^{q'}\vz\lf(B_k,\vbl \r)
\le\wz C_3\lf(\frac{\wz C_2\wz C_1\sz}\gz\r)^{q'}
\frac{\wz C_0}\sz\vz\lf(B,\vbl \r),\noz
\end{eqnarray*}
where $\wz C_3$ is the positive constant $C$ as in Lemma \ref{weight}.
Combining this with \eqref{jn2} and \eqref{jn3},
we see that, for all $\wz C_0\le \sz\le \az$,
$0<\gz<\gz$ and $B\st B_0$,
\begin{eqnarray}\label{jn6}
\lz(\az,B)&&\le\lf[\frac{\wz C_0\cf(\az-\gz)}\sz
+\wz C_3\lf(\frac{\wz C_2\wz C_1\sz}\gz\r)^{q'}
\frac{\wz C_0}\sz\r]\vz\lf(B,\vbl \r).
\end{eqnarray}
Now choose $\sz:=4^{q'}\wz C_0$, $\gz:=\az/2$
and $C_0:=\max\{\sz,\wz C_0\wz C_3(2\wz C_1\wz C_2)^{q'}\sz^{q'-1}\}$.
Then \eqref{jn6} implies that, for all $\az>C_0$,
\begin{eqnarray}\label{jn7}
\cf(\az)&&\le4^{-q'}\cf\lf(\frac\az2\r)
+C_0\az^{-q'}.
\end{eqnarray}
We now claim that, if $C_0<\az\le 2C_0$ and $m\in\zz_+$, then
\begin{eqnarray}\label{jn8}
\cf(2^m\az)&&\le(2C_0)^{q'}(2^m\az)^{-q'}.
\end{eqnarray}
Indeed, when $m=0$, it holds that $\cf(2^m\az)\le1\le(2C_0)^{q'}\az^{-q'}$
and hence \eqref{jn8} holds true in this case.
Assuming that \eqref{jn8} holds with $m$ replaced by $m-1$,
then from \eqref{jn7}, it follows that
\begin{eqnarray*}
\cf(2^m\az)
&&\le4^{-q'}\cf(2^{m-1}\az)+C_0(2^m\az)^{-q'}
\le4^{-q'}(2C_0)^{q'}(2^{m-1}\az)^{-q'}+C_0(2^m\az)^{-q'}\\
&&=(2C_0)^{q'}(2^m\az)^{-q'}(2^{-q'}+2^{-q'}C_0^{1-q'}).
\end{eqnarray*}
By this, together with the fact that $2^{-q'}+2^{-q'}C_0^{1-q'}<2^{-q'}+2^{-q'}<1$,
we know that \eqref{jn8} holds true for $m$.
Thus, by induction on $m$, we further conclude that the above claim holds true.
Moreover, by this claim, we further see that, if $\az>C_0$,
then $\cf(\az)\le (2C_0)^{q'}\az^{-q'}$,
which completes the proof of Theorem \ref{J-N}.
\end{proof}

\begin{rem}\label{r-JN}
(i) When $\vz(x,t):=t$ for all $x\in\rn$ and $t\in(0,\fz)$, and $s=0$,
then $\|\chi_B\|_\lv=|B|$ and hence the conclusion of
Theorem \ref{J-N} becomes that there exists a positive constant $C$ such that, for all balls $B\st\rn$,
$f\in\bmo(\rn)$ and $\az\in(0,\fz)$, it holds that
$$|\{x\in B:\ |f(x)-f_B|>\az\}|\le C e^{-\az/\|f\|_{\bmo(\rn)}}|B|,$$
which is the classical John-Nirenberg inequality
obtained by John and Nirenberg \cite{jn}.

(ii) When $\vz$ is as in Remark \ref{r-def}(i),
Theorem \ref{J-N} was proved by Li \cite{l08}.

(iii) When $\vz(x,t):=w(x)t$ for all $x\in\rn$ and $t\in(0,\fz)$,
$w\in A_{\fz}(\rn)$ and $s=0$,
Theorem \ref{J-N} is the John-Nirenberg inequality for the
weighted BMO space $\bmo_w(\rn)$,
which was obtained by Muckenhoupt and Wheeden \cite{mw76}.

(iv) When $\vz$ is as in Remark \ref{r-def}(ii)
with $p\in(0,1)$ and $w\in A_{\fz}(\rn)$, Theorem \ref{J-N} is new.
\end{rem}

Now, using Theorem \ref{J-N},
we establish some equivalent characterizations for $\lqs$.

\begin{thm}\label{char}
Let $s\in\zz_+$, $q\in[1,q(\vz)')$,
$\ez\in(n[\frac{q(\vz)}{i(\vz)}-1],\fz)$
and $\vz$ be a growth function.
Then, for all locally integrable functions $f$,
the following statements are mutually equivalent:
\begin{eqnarray*}
{\rm (i)}\ \|f\|_\lys&&:=\sup_{B\subset\rn}
\frac{1}{\|\chi_{B}\|_{L^{\vz}(\rn)}}
\int_{B}\lf|f(x)-\ps f(x)\r|\,dx<\fz;\\
{\rm (ii)}\ \|f\|_\lqs
&&:=\sup_{B\subset\rn}\frac{1}{\|\chi_{B}\|_{L^{\vz}(\rn)}}\\
&&\hs\times\lf\{
\int_{B}\lf[\frac{\lf|f(x)-\ps f(x)\r|}{\vz(x,\vb)}\r]^q\vz\lf(x,\vb\r)\,dx\r\}^{1/q}<\fz;\\
{\rm (iii)}\ \|f\|_{\wz{\cl_{\vz,q,s}}(\rn)}
&&:=\sup_{B\subset\rn}
\frac{1}{\|\chi_{B}\|_{L^{\vz}(\rn)}}\\
&&\hs\times\lf\{\inf_{p\in\cp_s(\rn)}
\int_{B}\lf[\frac{\lf|f(x)-p(x)\r|}{\vz(x,\vb)}\r]^q\vz\lf(x,\vb\r)\,dx\r\}^{1/q}<\fz;\\
{\rm (iv)}\ \|f\|_{\wz{\cl_{\vz,1,s}^\ez}(\rn)}
&&:=\sup_{B:=B(x_0,\dz)\subset\rn}
\frac{|B|}{\|\chi_{B}\|_{L^{\vz}(\rn)}}
\int_{\rn}
\frac{\dz^\ez\lf|f(x)-\ps f(x)\r|}{\dz^{n+\ez}+|x-x_0|^{n+\ez}}\,dx<\fz.
\end{eqnarray*}

Moreover,
$\|\cdot\|_\lys$,
$\|\cdot\|_\lqs$,
$\|\cdot\|_{\wz{\cl_{\vz,q,s}}(\rn)}$ and
$\|\cdot\|_{{\cl_{\vz,1,s}^\ez}(\rn)}$
are equivalent each other with the equivalent constants
independent of $f$.
\end{thm}

\begin{proof}
We first prove that (i) is equivalent to (ii).

By H\"older's inequality, for any ball $B\st\rn$ and $q\in(1,\fz)$,
we see that
\begin{eqnarray*}
&&\int_{B}\lf|f(x)-\ps f(x)\r|\,dx\\
&&\hs\le
\lf\{\int_{B}\lf[\frac{\lf|f(x)-\ps f(x)\r|}{\vz(x,\vb)}\r]^q
\vz\lf(x,\vb\r)\,dx\r\}^{1/q}\!\!
\lf\{\int_{B}\vz\lf(x,\vb\r)\,dx\r\}^{1/q'}\\
&&\hs=
\lf\{\int_{B}\lf[\frac{\lf|f(x)-\ps f(x)\r|}{\vz(x,\vb)}\r]^q
\vz\lf(x,\vb\r)\,dx\r\}^{1/q}.
\end{eqnarray*}
Thus, (ii) implies (i).

Conversely, if $\vz\in\aa_1(\rn)$, then $q(\vz)=1$.
By Theorem \ref{J-N},
for any $B\st\rn$ and $q\in(1,\fz)$,
we conclude that
\begin{eqnarray*}
&&\int_{B}\lf[\frac{\lf|f(x)-\ps f(x)\r|}{\vz(x,\vb)}\r]^q
\vz\lf(x,\vb\r)\,dx\\
&&\hs=q\int_0^\fz\az^{q-1}
\vz\lf(\lf\{x\in B:\ \frac{\lf|f(x)-\ps f(x)\r|}{\vz(x,\vb)}>\az\r\},\vb\r)
\,d\az\\
&&\hs\ls q\int_0^\fz\az^{q-1}
\exp\lf\{-\frac{C_2\az}{\|f\|_{\lys}\|\chi_B\|_\lv}\r\}\,d\az
\sim \|f\|_{\lys}^q\|\chi_B\|_\lv^q.
\end{eqnarray*}

If $\vz\notin\aa_1(\rn)$, then for any $r>q(\vz)$,
$\vz\in\aa_r(\rn)$ and
there exists $\ez\in(0,r-q(\vz))$ such that $\vz\in\aa_{r-\ez}(\rn)$.
Therefore, by Theorem \ref{J-N},
for any $B\st\rn$ and $q\in[1,(r-\ez)')$,
we see that
\begin{eqnarray*}
&&\int_{B}\lf[\frac{\lf|f(x)-\ps f(x)\r|}{\vz(x,\vb)}\r]^q
\vz\lf(x,\vb\r)\,dx\\
&&\hs=q\int_0^\fz\az^{q-1}
\vz\lf(\lf\{x\in B:\ \frac{\lf|f(x)-\ps f(x)\r|}{\vz(x,\vb)}>\az\r\},\vb\r)
\,d\az\\
&&\hs\ls q\int_0^\fz\az^{q-1}
\lf[1+\frac{\az}
{\|f\|_{\lys}\|\chi_B\|_\lv}\r]^{-(r-\ez)'}\,d\az
\sim \|f\|_{\lys}^q\|\chi_B\|_\lv^q,
\end{eqnarray*}
which implies that (ii) holds for all $q\in[1,q(\vz)')$.
Thus, (i) is equivalent to (ii).

Next we prove that (ii) is equivalent to (iii).
Obliviously, (ii) implies (iii).

Conversely, since $q\in[1,q(\vz)')$,
it follows that $\vz\in\aa_{q'}(\rn)$ and hence
\begin{eqnarray*}
&&\frac1{|B|^{q'}}
\lf\{\int_{B}\lf[\vz\lf(x,\vb\r)\r]^{1-q}\,dx\r\}^{q'/q}\\
&&\hs=\frac1{|B|^{q'}}\int_B\vz\lf(x,\vb\r)\,dx
\lf\{\int_{B}\lf[\vz\lf(x,\vb\r)\r]^{-1/(q'-1)}\,dx\r\}^{q'/q}\ls1,
\end{eqnarray*}
which, together with $\vz(B,\vb)=1$,
Lemma \ref{PBg} and H\"older's inequality,
further implies that,
for any $B\st\rn$, $q\in(1,\fz)$ and $p\in\cp_s(\rn)$,
\begin{eqnarray*}
&&\lf\{\int_{B}\lf[\frac{\lf|\ps(p-f)(x)\r|}{\vz(x,\vb)}\r]^q
\vz\lf(x,\vb\r)\,dx\r\}^{1/q}\\
&&\hs\ls\frac1{|B|}\int_B|p(x)-f(x)|\,dx
\lf\{\int_{B}\lf[\vz\lf(x,\vb\r)\r]^{1-q}\,dx\r\}^{1/q}\\
&&\hs\ls\lf\{\int_{B}\lf[\frac{\lf|f(x)-p(x)\r|}{\vz(x,\vb)}\r]^q
\vz\lf(x,\vb\r)\,dx\r\}^{1/q}\\
&&\hs\hs\times\lf\{\int_B\vz\lf(x,\vb\r)\,dx\r\}^{1/q'}\frac1{|B|}
\lf\{\int_{B}\lf[\vz\lf(x,\vb\r)\r]^{1-q}\,dx\r\}^{1/q}\\
&&\hs\ls\lf\{\int_{B}\lf[\frac{\lf|f(x)-p(x)\r|}{\vz(x,\vb)}\r]^q
\vz\lf(x,\vb\r)\,dx\r\}^{1/q}.
\end{eqnarray*}
Thus, from this, it follows that
\begin{eqnarray*}
&&\lf\{\int_{B}\lf[\frac{\lf|f(x)-\ps f(x)\r|}{\vz(x,\vb)}\r]^q
\vz\lf(x,\vb\r)\,dx\r\}^{1/q}\\
&&\hs\le\lf\{\int_{B}\lf[\frac{\lf|f(x)-p(x)\r|}{\vz(x,\vb)}\r]^q
\vz\lf(x,\vb\r)\,dx\r\}^{1/q}\\
&&\hs\hs+\lf\{\int_{B}\lf[\frac{\lf|\ps(p-f)(x)\r|}{\vz(x,\vb)}\r]^q
\vz\lf(x,\vb\r)\,dx\r\}^{1/q}\\
&&\hs\ls\lf\{\int_{B}\lf[\frac{\lf|f(x)-p(x)\r|}{\vz(x,\vb)}\r]^q
\vz\lf(x,\vb\r)\,dx\r\}^{1/q}.
\end{eqnarray*}
Namely, (iii) implies (ii) and hence (ii) is equivalent to (iii).

Finally we prove that (iv) is equivalent to (i).
Obviously, (iv) implies (i).

Conversely,
for any $k\in\zz_+$, let $B_k:=2^kB$.
Then, for all $k\in\zz_+$ and $x\in B_k$, by Lemma \ref{PBg}, we have
\begin{eqnarray}\label{5.1}
|P_{B_{k+1}}^sf(x)-P_{B_k}^sf(x)|
&&=|P_{B_k}^s(f-P_{B_{k+1}}^sf)(x)|\\
&&\le\frac{2^n}{|B_{k+1}|}
\int_{B_{k+1}}|f(x)-P_{B_{k+1}}^sf(x)|\,dx\noz\\
&&\le2^n
\frac{\|\chi_{B_{k+1}}\|_{L^\vz(\rn)}}{|B_{k+1}|}
\|f\|_{\lys}.\noz
\end{eqnarray}
Since $\ez\in(n[\frac{q(\vz)}{i(\vz)}-1],\fz)$, it follows that
there exist $p_0\in(0,i(\vz))$ and $q_0\in(q(\vz),\fz)$ such that
$\ez>n(\frac{q_0}{p_0}-1)$. Thus, $\vz\in\aa_{q_0}(\rn)$ and $\vz$ is of uniformly lower type $p_0$, which further
implies that, for all $j\in\zz_+$,
$$\vz\lf(B_j, 2^{-jnq_0/p_0}\|\chi_{B}\|^{-1}_{L^\vz(\rn)}\r)\ls2^{-jnq_0}
\vz\lf(B_j,\|\chi_{B}\|^{-1}_{L^\vz(\rn)}\r)\ls1.
$$
From this, we deduce that, for all $j\in\zz_+$, $\|\chi_{B_j}\|_{L^\vz(\rn)}
\ls2^{jnq_0/p_0}\|\chi_{B}\|_{L^\vz(\rn)}$,
which, together with \eqref{5.1}, implies that, for all $k\in\nn$,
\begin{eqnarray*}
|P_{B_k}^sf(x)-P_B^sf(x)|&\le&
\sum_{j=1}^k|P_{B_j}^sf(x)-P_{B_{j-1}}^sf(x)|
\le
2^n\|f\|_{\lys}\sum_{j=1}^k\frac{\|\chi_{B_{j}}\|_{L^\vz(\rn)}}{|B_{j}|}
\\
&\ls&\lf\{\sum_{j=1}^{k}2^{jn(q_0/p_0-1)}\r\}
\frac{\|\chi_{B}\|_{L^\vz(\rn)}}{|B|}
\|f\|_{\lys}\\ \nonumber
&\ls&2^{kn(q_0/p_0-1)}\frac{\|\chi_{B}\|_{L^\vz(\rn)}}{|B|}
\|f\|_{\lys}.
\end{eqnarray*}
By this, we conclude that
\begin{eqnarray*}
&&\int_\rn\frac{\dz^\ez|f(x)-P_B^sf(x)|}{\dz^{n+\ez}
+|x-x_0|^{n+\ez}}\,dx\\
&&\hs\le \int_{B}\frac{\dz^\ez|f(x)-P_B^sf(x)|}{\dz^{n+\ez}
+|x-x_0|^{n+\ez}}\,dx+\sum_{k=0}^\fz\int_{B_{k+1}\setminus B_{k}}\cdots\\
&&\hs\ls\frac1{|B|}\int_{B}|f(x)-P_B^sf(x)|\,dx
+\sum_{k=1}^\fz(2^k\dz)^{-(n+\ez)}\dz^\ez\int_{B_k}
|f(x)-P_B^sf(x)|\,dx\\
&&\hs\ls\frac{\|\chi_{B}\|_{L^\vz(\rn)}}{|B|}
\|f\|_{\lys}\\
&&\hs\hs+\sum_{k=1}^\fz2^{-k(n+\ez)}\frac1{|B|}
\int_{B_k}\lf[|f(x)-P_{B_k}^sf(x)|\,dx+|P_{B_k}^sf(x)-P_B^sf(x)|\r]\,dx\\
&&\hs\ls\lf\{\sum_{k=1}^\fz2^{-k(n+\ez-nq_0/p_0)}\r\}
\frac{\|\chi_{B}\|_{L^\vz(\rn)}}{|B|}
\|f\|_{\lys}\ls\frac{\|\chi_{B}\|_{L^\vz(\rn)}}{|B|}
\|f\|_{\lys},
\end{eqnarray*}
which completes the proof of Theorem \ref{char}.
\end{proof}

\begin{rem}\label{r-char}
(i) When $\vz$ is as in Remark \ref{r-def}(i),
Theorem \ref{char} was proved by Taibleson and Weiss \cite{tw80}.

(ii) When $\vz$ is as in Remark \ref{r-def}(ii) with $w\in A_{1}(\rn)$,
Theorem \ref{char} was obtained in \cite{yy11}.

(iii) When $\vz$ is as in Remark \ref{r-def}(ii) with $p\in(0,1)$,
Theorem \ref{char} is new.
\end{rem}


\section{Dual Spaces of Musielak-Orlicz Hardy Spaces\label{s3}}

\hskip\parindent In this section, we prove that the dual space of $\hv$
is $\lqs$ for all $q\in[1,q(\vz)')$ and $s\in[m(\vz),\fz)\cap\zz_+$.

In what follows, we denote by $\cs(\rn)$ the
\emph{space of all Schwartz functions} and
by $\cs'(\rn)$ its \emph{dual space}
(namely, the \emph{space of all tempered distributions}).
For $m\in\nn$, let
$$\cs_m(\rn)
:=\lf\{\psi\in\cs(\rn):\ \sup_{x\in\rn}\sup_{\bz\in\zz^n_+,\,|\bz|\le m+1}
(1+|x|)^{(m+2)(n+1)}|\partial^\bz_x\psi(x)|\le1\r\}.$$
Then for all $f\in\cs'(\rn)$,
the \emph{nontangential grand maximal function} $f^\ast_m$
of $f$ is defined by setting, for all $x\in\rn$,
\begin{equation*}
f^\ast_m(x):=\sup_{\psi\in\cs_m(\rn)}\sup_{|y-x|<t,\,t\in(0,\fz)}|f\ast\psi_t(y)|,
\end{equation*}
where for all $t\in(0,\fz)$, $\psi_t(\cdot):=t^{-n}\psi(\frac{\cdot}{t})$.
When $m(\vz):=\lfz n[q(\vz)/i(\vz)-1]\rfz$, where $q(\vz)$ and $i(\vz)$ are,
respectively, as in \eqref{qvz} and \eqref{ivz},
we \emph{denote $f^\ast_{m(\vz)}$ simply
by $f^\ast$}.

Now we recall the definition of the Musielak-Orlicz Hardy space $H^\vz(\rn)$ introduced
by Ky \cite{ky} as follows.

\begin{defn}\label{d-hd}
Let $\vz$ be a growth function. The \emph{Musielak-Orlicz Hardy space
$H^\vz(\rn)$} is defined to be the space of all $f\in\cs'(\rn)$ such that
$f^\ast\in L^\vz(\rn)$
with the \emph{quasi-norm}
$\|f\|_{H^\vz(\rn)}:=\|f^\ast\|_{L^\vz(\rn)}$.
\end{defn}

In order to prove our main result,
we need to introduce the atomic Musielak-Orlicz Hardy space.
For any ball $B$ in $\mathbb R^n$, the \emph{space $L_\varphi^q(B)$} for
$q\in[1,\infty]$ is defined to be the set of all measurable functions $f$ on
$\mathbb R^n$ supported in $B$ such that
\begin{eqnarray*}
\; \|f\|_{L_\varphi^q(B)}:= \left\lbrace
\begin{array}{l l}
\displaystyle \sup_{t\in(0,\fz)}\lf[\frac1{\varphi(B,t)}{\int_{\mathbb
R^n}|f(x)|^q\varphi(x,t)dx}\r]^{1/q}<\infty,
\ &q\in [1,\fz);\\
\\
\|f\|_{L^\infty(\rn)}<\infty,
&q=\infty.
\end{array} \right.
\end{eqnarray*}

Now, we recall the atomic Musielak-Orlicz Hardy spaces introduced by Ky \cite{ky}
as follows. A triplet $(\varphi,q,s)$ is said to be \emph{admissible}, if
$q\in(q(\varphi),\infty]$ and $s\in\mathbb N$ satisfies $s\geq
m(\varphi)$. A measurable function $a$ is called a \emph{$(\varphi,q,s)$-atom}
if it satisfies the following three conditions:

{\rm (i)} $a\in L_\varphi^q(B)$ for some ball $B$;

{\rm (ii)} $\|a\|_{L_\varphi^q(B)}\le \|\chi_B\|_{L^\varphi(\rn)}^{-1}$;

{\rm (iii)} $\int_{\mathbb R^n}a(x)x^\alpha dx=0$ for any $|\alpha|\le s$,
where $\az:=(\az_1,\ldots,\az_n)\in\zz_+^n$
and $|\az|:=\az_1+\cdots+\az_n$.

The \emph{atomic Musielak-Orlicz Hardy space}
$H_{\mathrm{at}}^{\varphi,q,s}(\mathbb R^n)$ is defined to be the space of all
$f\in\mathcal{S}'(\mathbb R^n)$ that can be
represented as a sum of multiples of $(\varphi,q,s)$-atoms, that
is, $f=\sum_{j=1}^\fz b_j$ in $\mathcal{S}'(\mathbb R^n)$,
where, for each $j$, $b_j$ is a multiple of some $(\varphi,q,s)$-atom supported in
some ball $B_j$, with the property
$\sum_{j=1}^\fz\varphi(B_j,\|b_j\|_{L_\varphi^q(B_j)})<\infty.$
For any given sequence of multiples of $(\varphi,q,s)-$atoms,
$\{b_j\}_{j\in\nn}$, let
$$\Lambda_q(\{b_j\}_{j\in\nn}):=\inf\lf\{\lambda>0:\ \sum_{j=1}^\fz
\varphi\lf(B_j,\frac{\|b_j\|_{L_\varphi^q(B_j)}}{\lambda}\r)\le 1\r\}$$
and then define
$$ \|f\|_{H_{\mathrm{at}}^{\varphi,q,s}(\rn)}
:=\inf\lf\{\Lambda_q(\{b_j\}_{j\in\nn}):\
f=\sum_{j=1}^\fz b_j\quad\text{in }\,\ \mathcal{S}'(\mathbb R^n)\r\},$$
where the infimum is taken over all decompositions of $f$ as above.
We use $H^{\vz,q,s}_{\fin}(\rn)$ to denote the \emph{set of
all finite combinations of $(\vz,q,s)$-atoms}.
The norm of $f$ in $\hvf$ is defined by
$$ \|f\|_\hvf
:=\inf\lf\{\Lambda_q(\{b_j\}_{j=1}^k):\
f=\sum_{j=1}^kb_j\quad\text{in }\,\ \mathcal{S}'(\mathbb R^n)\r\},$$
where the infimum is taken over all finite decompositions of $f$.
It is easy to see that
$\hvf$ is dense in $\hva$.

In order to obtain the finite atomic decomposition,
Ky \cite{ky} introduced a
\emph{uniformly locally dominated convergence condition}
as follows:

Let $K$ be a compact set in $\rn$.
Let $\{f_m\}_{m\in\nn}$ be a sequence of measurable functions
such that $f_m(x)$ tends to $f(x)$ for almost every $x\in\rn$ as $m\to\fz$.
If there exists a nonnegative measurable function $g$ such that
$|f_m(x)|\le g(x)$ for all $m\in\nn$ and almost every $x\in\rn$, and
$\sup_{t>0}\int_Kg(x)\frac{\vz(x,t)}{\int_K\vz(y,t)\,dy}\,dx<\fz,$
then
$\sup_{t>0}\int_K|f_m(x)-f(x)|\frac{\vz(x,t)}{\int_K\vz(y,t)\,dy}\,dx$
tends $0$ as $m\to\fz$.

Observe that the growth functions $\vz(x,t):=w(x)\Phi(x)$,
with $w\in A_\fz(\rn)$ and $\Phi$ being an Orlicz function,
and $\vz(x,t)=\frac{t^p}{[\log(e+|x|)+\log(e+t^p)]^p}$, with $p\in(0,1]$
for all $x\in\rn$ and $t\in(0,\fz)$,
satisfy the uniformly locally dominated convergence condition.

The following Lemmas \ref{atom}, \ref{at-ch} and \ref{at-fin}
are just \cite[Lemma 4.4, Theorems 3.1 and Theorem 3.4]{ky}.

\begin{lem}\label{atom}
Let $(\varphi,q,s)$ be admissible.
Then there exists a positive constant $C$
such that, for all $f=\sum_{j=1}^\fz b_j\in\hva$,
$$\sum_{j=1}^\fz\|b_j\|_{L_\vz^q(B_j)}\|\chi_{B_j}\|_\lv
\le C\blz_q(\{b_j\}_{j\in\nn}),$$
where for any $j\in\nn$, $b_j$ is a multiple of
$(\vz, q,s)$-atom associated with the ball $B_j$.
\end{lem}

\begin{lem}\label{at-ch}
Let $(\varphi,q,s)$ be admissible. Then $H^\varphi(\mathbb
R^n)=H_{\mathrm{at}}^{\varphi,q,s}(\mathbb R^n)$ with equivalent norms.
\end{lem}

\begin{lem}\label{at-fin}
Let $\vz$ be a growth function satisfying uniformly locally dominated
convergence condition, $(\vz,q,s)$ admissible and $q\in(q(\vz),\fz)$.
Then $\|\cdot\|_\hvf$ and $\|\cdot\|_\hv$ are equivalent
quasi-norms on $\hvf$.
\end{lem}

\begin{thm}\label{dual}
Let $\varphi$ be a growth function
satisfying uniformly locally dominated convergence condition
and $s\in[m(\vz),\fz)\cap\zz_+$.
Then the dual space of $\hv$, denoted by $(\hv)^\ast$,
is $\lys$ in the following sense:

{\rm (i)} Suppose that $b\in \lys$. Then the linear functional
$L_b:\ f\to L_b(f):=
\int_\rn f(x)b(x)\,dx$, initially defined for all
$f\in H^{\vz,q,s}_{\fin}(\rn)$
with some $q\in(q(\vz),\fz)$,
has a bounded extension to $\hv$.

{\rm (ii)} Conversely,
every continuous linear functional on $\hv$ arises as in {\rm(i)}
with a unique $b\in\lys$.

Moreover, $\|b\|_{\lys}\sim\|L_b\|_{(\hv)^\ast}$,
where the implicit constants are independent of $b$.
\end{thm}

\begin{proof}
By Theorem \ref{char} and Lemma \ref{at-ch},
to prove $\lys\st(\hv)^*$,
it is sufficient to show $\lqps\st(\hva)^*$.
Let $g\in\lqps$, $a$ be a $(\vz,q,s)$-atom
associated with a ball $B\st\rn$.
Then by the moment and size conditions of the atom $a$,
together with the H\"older's inequality,
we see that
\begin{eqnarray*}
\lf|\int_{\rn}a(x)g(x)\,dx\r|
&&=\lf|\int_{\rn}a(x)[g(x)-P_B^sg(x)]\,dx\r|\noz\\
&&\le\|a\|_\lvq
\lf\{\int_{\rn}\lf[\frac{\lf|{g(x)-P_B^sg(x)}\r|}{\vz(x,\vb)}\r]^{q'}
\vz\lf(x,\vb\r)\,dx\r\}^{1/q'}\noz\\
&&\le\frac1{\|\chi_B\|_\lv}
\lf\{\int_{\rn}\lf[\frac{\lf|{g(x)-P_B^sg(x)}\r|}{\vz(x,\vb)}\r]^{q'}
\vz\lf(x,\vb\r)\,dx\r\}^{1/q'}\noz\\
&&=\|g\|_\lqps.\noz
\end{eqnarray*}
Thus, by Lemma \ref{atom},
for a sequence $\{b_j\}_{j\in\nn}$ of multiples of
$(\vz, q,s)$-atoms associated with balls $\{B_j\}_{j\in\nn}$ and
$f=\sum_{j=1}^m b_j\in\hva$,
we have
\begin{eqnarray*}
\lf|\int_{\rn}f(x)g(x)\,dx\r|
&&\le\sum_{k=1}^m\|b_j\|_{L_\vz^q(B_j)}\|\chi_{B_j}\|_\lv\|g\|_\lqps\\
&&\ls \blz_q(\{b_j\}_{j=1}^m)\|g\|_\lqps,
\end{eqnarray*}
which, together with Lemma \ref{at-fin} and the fact that
$H^{\vz,q,s}_{\fin}(\rn)$ is dense in $H^{\vz,q,s}_{\rm at}(\rn)$,
completes the proof of (i).

Conversely, suppose $L\in(\hv)^*=(\hva)^*$,
where $(\vz,q,s)$ is admissible.
For a ball $B$ in $\rn$ and $q\in(q(\vz),\fz]$,
let
$$L_{\vz,s}^q(B)
:=\lf\{f\in\lvq:\ \int_\rn f(x)x^\az\,dx=0\
\mbox{ for all } \az\in\zz_+^n \mbox{ and }|\az|\le s\r\}.$$
Then, $\lvb\st\hv$ and, for all $f\in\lvb$,
$a:=\|\chi_B\|_\lv^{-1}\|f\|_\lvq^{-1}f$ is a $(\vz,q,s)$-atom
and hence $\|f\|_\hva\le\|\chi_B\|_\lv\|f\|_\lvq$.
Thus, for all $L\in(\hva)^*$ and $f\in\lvb$,
$$|Lf|\le\|L\|\|f\|_\hva.$$
Therefore, $L$ is a bounded linear functional on $\lvb$
which,  by the Hahn-Banach theorem, can be extended to the whole
space $\lvq$ without increasing its norm.
By this, together with the Lebesgue-Nikodym theorem,
we conclude that there exists $h\in L^1(B)$ such that, for all $f\in\lvb$,
$$L(f)=\int_\rn f(x)h(x)\,dx.$$

We now take a sequence of balls $\{B_j\}_{j\in\nn}$
such that $B_1\st B_2\st\cdots\st B_j\st\cdots$
and $\cup_{j=1}^\fz B_j=\rn$.
Then, by the above argument,
we see that there exists a sequence of $\{h_j\}_{j\in\nn}$
such that, for all $j\in\nn$,
$h_j\in L^1(B_j)$ and, for all $f\in L_{\vz,s}^q(B_j)$,
\begin{equation}\label{3.x1}
L(f)=\int_\rn f(x)h_j(x)\,dx.
\end{equation}
Thus, for all $f\in L_{\vz,s}^q(B_1)$,
$$\int_{B_1}f(x)[h_1(x)-h_2(x)]\,dx=0,$$
which, together with the fact that $g-P_{B_1}^sg\in L_{\vz,s}^q(B_1)$
for all $g\in L_\vz^q(B_1)$,
further implies that, for all $g\in L_\vz^q(B_1)$,
$$\int_{B_1}[g(x)-P_{B_1}^sg(x)][h_1(x)-h_2(x)]\,dx=0.$$
By
\begin{eqnarray*}
&&\int_BP_B^sg(x)f(x)-P_B^sf(x)g(x)\,dx\\
&&\hs=\int_BP_B^sg(x)[f(x)-P_B^sf(x)]+P_B^sf(x)[P_B^sg(x)-g(x)]\,dx=0,
\end{eqnarray*}
we have
$$\int_{B_1}P_{B_1}^sg(x)[h_1(x)-h_2(x)]\,dx
=\int_{B_1}g(x)P_{B_1}^s(h_1-h_2)(x)\,dx.
$$
Thus, for all $g\in L_\vz^q(B_1)$,
$$\int_{B_1}g(x)[h_1(x)-h_2(x)-P_{B_1}^s(h_1-h_2)(x)]\,dx=0,$$
which implies that, for almost every $x\in B_1$,
$(h_1-h_2)(x)=P_{B_1}^s(h_1-h_2)(x)$.

Let $\wz h_1:=h_1$ and, for $j\in\nn$,
$\wz h_{j+1}:=h_{j+1}+P_{B_j}(\wz h_j-h_{j+1})$.
Then we have a new sequence $\{\wz h_j\}_{j\in\nn}$ satisfying that,
for almost every $x\in B_j$, $\wz h_{j+1}(x)=\wz h_j(x)$ and
$\wz h_j\in L^1(B_j)$.
Let $b$ be a measurable function satisfying that, if $x\in B_j$,
$b(x)=\wz h_j(x)$.
It remains to prove that $b\in\lys$ and, for all $f\in\hvf$,
$$L(f)=\int_\rn f(x)b(x)\,dx.$$

For any $f\in\hvf$, it is easy to see that there exists $j\in\nn$ such that $\supp f\st B_j$.
Thus, $f\in L_{\vz,s}^q(B_j)$ and, by \eqref{3.x1}, we further see that
$$L(f)=\int_\rn f(x)b(x)\,dx.$$

For any ball $B\st\rn$,
let $f:={\rm sign}(b-P_B^sb)$ and
$a:=\frac12\|\chi_B\|_\lv^{-1}(f-P_B^sf)\chi_B$.
Then $a$ is a $(\vz,q,s)$-atom and
\begin{eqnarray*}
\frac1{\|\chi_B\|_\lv}
\int_B\lf|b(x)-P_B^sb(x)\r|\,dx
&&=\frac1{\|\chi_B\|_\lv}
\lf|\int_B[b(x)-P_B^sb(x)]f(x)\,dx\r|\\
&&=\frac1{\|\chi_B\|_\lv}
\lf|\int_Bb(x)[f(x)-P_B^sf(x)]\,dx\r|\ls|L(a)|\\
&&\ls\|L\|_{(\hva)^*}\|a\|_\hv
\ls\|L\|_{(\hv)^*}.
\end{eqnarray*}
Thus, $b\in\lys$ and $\|b\|_\lys\ls\|L\|_{(\hv)^*}$,
which completes the proof of Theorem \ref{dual}.
\end{proof}

\begin{rem}\label{r-dual}
(i) When $\vz$ is as in Remark \ref{r-def}(i),
Theorem \ref{dual} was proved by Taibleson and Weiss \cite{tw80}.

(ii) When $\vz$ is as in Remark \ref{r-def}(ii),
Theorem \ref{dual} was obtained by Garc\'ia-Cuerva \cite{g79}.
\end{rem}

From Theorems \ref{char} and \ref{dual}, we immediately deduce the following
interesting conclusion.

\begin{cor}\label{c3.1}
Let $\vz$ be a growth function satisfying
uniformly locally dominated convergence condition.
Then, for all $q\in[1,q(\vz)')$ and $s\in [m(\vz),\fz)\cap\zz_+$,
$\lqs$ and $\cl_{\vz, 1, m(\vz)}(\rn)$ coincide with equivalent norms.
\end{cor}


\section{The $\vz$-Carleson Measure Characterization of $\lys$\label{s4}}

\hskip\parindent
In this section, we establish the $\vz$-Carleson
measure characterization of $\lys$.
We first introduce the following $\vz$-Carleson
measures.

\begin{defn}\label{d-car}
Let $\vz$ be a growth function. A measure $d\mu$ on $\rr^{n+1}_+$
is called a \emph{$\vz$-Carleson measure} if
$$\|d\mu\|_{\vz}
:=\sup_{B\subset\rn}\frac{1}
{\|\chi_B\|_{L^\vz(\rn)}}
\lf\{\int_{\widehat{B}}
\frac{t^n}{\vz(B(x,t),\vb)}
\,\lf|d\mu(x,t)\r|\r\}^{1/2}<\fz,
$$
where the supremum is taken over all balls $B:=B(x_0,r)\subset\rn$ and
$$\widehat{B}:=\{(x,t)\in\rr^{n+1}_+:\ |x-x_0|+t<r\}$$
denotes the \emph{tent} over $B$.
\end{defn}

To obtain the $\vz$-Carleson measure characterization of $\lys$,
we need to recall the Musielak-Orlicz tent space
introduced in \cite{hyy}.
Let $\rr^{n+1}_+:=\rn\times(0,\fz)$. For any $x\in\rn$, let
$$\bgz(x):=\{(y,t)\in\rr^{n+1}_+:\ |x-y|<t\}
$$
be the \emph{cone of aperture $1$ with vertex $x\in\rn$}.

For all measurable functions $g$ on $\rr^{n+1}_+$ and $x\in\rn$, define
$$\ca(g)(x):=\lf\{\int_{\bgz(x)}|g(y,t)|^2
\frac{dy\,dt}{t^{n+1}}\r\}^{1/2}.
$$
Recall that a
measurable function $g$ is said
to belong to the \emph{tent space} $T^p_2(\rr^{n+1}_+)$ with $p\in(0,\fz)$,
if $\|g\|_{T^p_2(\rr^{n+1}_+)}:=\|\ca(g)\|_{L^p(\rn)}<\fz$.

Let $\vz$ be as in Definition \ref{d-vz}. In what follows, we denote by
$T_{\vz}(\rr^{n+1}_+)$ the \emph{space} of all measurable functions
$g$ on $\rr^{n+1}_+$ such that $\ca(g)\in L^{\vz}(\rn)$ and, for any
$g\in T_{\vz}(\rr^{n+1}_+)$, define its \emph{quasi-norm} by
$$\|g\|_{T_{\vz}(\rr^{n+1}_+)}:=\|\ca(g)\|_{L^{\vz}(\rn)}=
\inf\lf\{\lz\in(0,\fz):\
\int_{\rn}\vz\lf(x,\frac{\ca(g)(x)}{\lz}\r)\,dx\le1\r\}.$$

Let $p\in(1,\fz)$. A function $a$ on $\rr^{n+1}_+$ is called a
\emph{$(\vz,\,p)$-atom} if

(i) there exists a ball $B\subset\rn$ such that $\supp
a\subset\widehat{B}$;

(ii) $\|a\|_{T^p_2(\rr^{n+1}_+)}\le
|B|^{1/p}\|\chi_B\|_{L^\vz(\rn)}^{-1}$.

Furthermore, if $a$ is a $(\vz,p)$-atom for all $p\in (1,\fz)$,
we then call $a$ a \emph{$(\vz,\fz)$-atom}.

On the space $\lys$, we have the following $\vz$-Carleson measure characterization.

\begin{thm}\label{carl}
Let $\vz$ be a growth function, $s\in[m(\vz),\fz)\cap\zz_+$,
where $q(\vz)$ and $i(\vz)$ are respectively as in \eqref{qvz} and \eqref{ivz},
$\vz\in\aa_1(\rn)$,
$\phi\in\mathcal{S}(\mathbb R^n)$ be a radial function,
$\supp\phi\subset\{x\in\mathbb R^n:\ |x|\le 1\},$
$\int_{\mathbb R^n}\phi(x) x^\gz \,dx=0$ for all $|\gz|\le s$
and, for all $\xi\in\rn\bh\{0\}$,
$$\int_0^\infty|\hat{\phi}(\xi t)|^2\frac{dt}{t}=1.$$
Then $b\in\lys$ if and only if
$b\in L_\loc^2(\rn)$ and,
for all $(x,t)\in\rr^{n+1}_+$,
$$d\mu(x,t):=|\phi_t*b(x)|^2\frac{dxdt}{t}$$
is a $\vz$-Carleson measure
on $\rr^{n+1}_+$. Moreover, there exists a
positive constant $C$, independent of $b$, such that
$\frac1C\|b\|_{\lys}\le \|d\mu\|_{\vz} \le C\|b\|_{\lys}$.
\end{thm}

\begin{proof}
Let $b\in\lys$ and $B_0:=B(x_0,r)\st \rn$.
Then,
\begin{eqnarray}\label{car4}
b=\psl b+(b-\psl b)\chi_{2B_0}+
(b-\psl b)\chi_{\rn\setminus2B_0}=:b_1+b_2+b_3.
\end{eqnarray}

For $b_1$, since $\int_\rn\phi(x)x^\gz\,dx=0$ for any $|\gz|\le s$,
we see that, for
all $t\in(0,\fz)$, it holds that $\phi_t\ast b_1\equiv0$ and hence
\begin{eqnarray}\label{car5}
\int_{\widehat{B_0}}|\phi_t*b_1(x)|^2
\frac{t^n}{\vz(B(x,t),\vbl)}\frac{dx\,dt}{t}=0.
\end{eqnarray}

For $b_2$,
by H\"older's inequality, for all balls $B\st\rn$ and $\tz\in(0,\fz)$,
we know that
\begin{eqnarray}\label{car1}
|B|=\int_B[\vz(x,\tz)]^{1/2}[\vz(x,\tz)]^{-1/2}\,dx
\le[\vz(B,\tz)]^{1/2}[\vz^{-1}(B,\tz)]^{1/2},
\end{eqnarray}
where above and in what follows, for any measurable set $E\st\rn$
and $\tz\in(0,\fz)$, we let
$\vz^{-1}(E,\tz):=\int_E[\vz(x,\tz)]^{-1}\,dx$.
From \eqref{car1}, it follows that
\begin{eqnarray}\label{car2}
&&\int_{\widehat{B_0}}|\phi_t*b_2(x)|^2
\frac{t^n}{\vz(B(x,t),\vbl)}\frac{dx\,dt}{t}\\
&&\hs\ls\int_{\widehat{B_0}}|\phi_t*b_2(x)|^2
\int_{B(x,t)}[\vz(y,\vbl)]^{-1}\,dy\dxtn\noz\\
&&\hs\ls\int_{B}[\vz(y,\vbl)]^{-1}\int_{\bgz(y)}|\phi_t*b_2(x)|^2
\dxtn\,dy\noz.
\end{eqnarray}
Since $\vz\in\aa_1(\rn)\st\aa_2(\rn)$,
it follows that $[\vz(\cdot,\vbl)]^{-1}\in A_2(\rn)$
(the class of Muckenhoupt weights).
By this, \eqref{car2}, Theorem \ref{char} and the boundedness
of the square function
$S_\phi f(y):=\int_{\bgz(y)}|\phi_t*b_2(x)|^2\dxtn$
on the weighted Lebesgue $L^2(\rn)$ space with the weight
$[\vz(\cdot,\vbl)]^{-1}$ (see, for example, \cite{gr85, st89, l11}),
we have
\begin{eqnarray}\label{car3}
\qquad&&\int_{\widehat{B_0}}|\phi_t*b_2(x)|^2
\frac{t^n}{\vz(B(x,t),\vbl)}\frac{dx\,dt}{t}\\
&&\hs\ls\int_{\rn}|b_2(y)|^2\lf[\vz\lf(y,\vbl\r)\r]^{-1}\,dy\noz\\
&&\hs\sim\int_{2B_0}|b(y)-\psl b(y)|^2\lf[\vz\lf(y,\vbl\r)\r]^{-1}\,dy\noz\\
&&\hs\ls\int_{2B_0}|b(y)-\psel b(y)|^2\lf[\vz\lf(y,\vbl\r)\r]^{-1}\,dy\noz\\
&&\hs\hs+\int_{2B_0}|\psel b(y)-\psl b(y)|^2\lf[\vz\lf(y,\vbl\r)\r]^{-1}\,dy\noz\\
&&\hs\ls\|\chi_{B_0}\|_\lv^2\|b\|_\lys^2,\noz
\end{eqnarray}
where the last inequality is deduced from $\vz\in\aa_1(\rn)$,
$\vz(2B_0,\vbl)\sim1$ and, for $y\in 2B_0$,
\begin{eqnarray*}
|\psel b(y)-\psl b(y)|
&&=|\psl(b-\psel b)(y)|\\
&&\ls \frac1{|B_0|}\int_{2B_0}|b(x)-\psel b(x)|\,dx
\ls\frac{\|\chi_{2B_0}\|_\lv}{|B_0|}\|b\|_\lys.
\end{eqnarray*}

Now, for $b_3$, let $B_k:=B(x_0, 2^kr)$.
By Lemma \ref{PBg}, Theorem \ref{char} and
$\phi\in\cs(\rn)$,
we conclude that, for all $x\in B_0$,
\begin{eqnarray*}
|\phi_t*b_3(x)|
&&\ls \int_{(\wz{B})^\complement}
\frac{t^\ez|b(y)-\pse b(y)|}{|y-x_B|^{n+\ez}}\,dy\ls\frac{t^\ez}{r^\ez}
\frac{\|\chi_{B_0}\|_{L^\vz(\rn)}}{|B_0|}
\|b\|_{\lys},
\end{eqnarray*}
which, together with \eqref{car1},
$\vz\in\aa_1(\rn)$ and $\vz(B_0,\vbl)=1$, implies that
\begin{eqnarray*}
&&\int_{\widehat{B_0}}|\phi_t*b_3(x)|^2
\frac{t^n}{\vz(B(x,t),\vbl)}\dxt\\
&&\hs\ls
\int_{\wh {B_0}}\frac{t^{2\ez}}{r^{2\ez}}\vz^{-1}\lf(B(x,t),\vbl\r)\dxtn
\frac{\|\chi_{B_0}\|_{L^\vz(\rn)}^2}{|B_0|^2}\|b\|_{\lys}^2\\
&&\hs\ls
\int_0^r\frac{t^{2\ez}}{r^{2\ez}}\dtn
\frac{\vz^{-1}(B_0,\vbl)}{|B_0|}\|\chi_{B_0}\|_{L^\vz(\rn)}^2\|b\|_{\lys}^2\\
&&\hs\ls\|\chi_{B_0}\|_{L^\vz(\rn)}^2\|b\|_{\lys}^2.
\end{eqnarray*}
From this, \eqref{car4}, \eqref{car5} and \eqref{car3}, we deduce that
$$\frac{1}{\|\chi_{B_0}\|_{L^\vz(\rn)}}
\lf\{\int_{\widehat{B_0}}|\phi_t*b(x)|^2
\frac{t^n}{\vz(B(x,t),\vbl)}\dxt\r\}^{1/2}\ls\|b\|_{\lys},$$
which, together with the arbitrariness of $B_0\subset\rn$, implies
that $d\mu(x,t):=|\phi_t*b(x)|^2\frac{dxdt}{t}$ for all $x\in\rn$ and $t\in (0,\fz)$
is a $\vz$-Carleson measure on $\rr^{n+1}_+$ and
$\|d\mu\|_{\vz}\ls\|b\|_{\lys}$.

Conversely, let $f\in H^{\vz,\,\fz,\,s}_{\fin}(\rn)$.
Then by $f\in L^\fz(\rn)$ with compact support, $b\in L^2_\loc(\rn)$
and the Plancherel formula, we conclude that
\begin{eqnarray*}
\int_{\rn}f(x)\overline{b(x)}\,dx=\int_{\rr^{n+1}_+}\phi_t\ast
f(x)\overline{\phi_t\ast b(x)}\,\frac{dx\,dt}{t},
\end{eqnarray*}
where $\overline{b(x)}$ and $\overline{\phi_t\ast b(x)}$ denote, respectively,
the conjugates of $b(x)$ and $\phi_t\ast b(x)$.
Moreover, from $f\in H^{\vz,\,\fz,\,s}_{\fin}(\rn)$
and \cite[Theorem 4.11]{hyy}, it follows that
$\phi_t\ast f\in T_{\vz}(\rr^{n+1}_+)$.
By this and \cite[Theorem 3.1]{hyy},
we know that there exist $\{\lz_j\}_{j\in\nn}\subset\cc$ and a sequence
$\{a_j\}_{j\in\nn}$ of $(\vz,\fz)$-atoms such that
$\phi_t\ast f=\sum_{j}\lz_ja_j$ almost everywhere
and $\sum_{j=1}^\fz|\lz_j|\ls\|f\|_\hv$.
From this, the Lebesgue dominated convergence theorem,
H\"older's inequality and $\vz\in\aa_1(\rn)$,
we deduce that
\begin{eqnarray*}
\lf|\int_\rn f(x)\overline{b(x)}\,dx\r|
&&\le\sum_{j=1}^\fz|\lz_j|\int_{\rr^{n+1}_+}|a_j(x,t)|
|\phi_t\ast b(x)|\,\frac{dx\,dt}{t}\\
&&\le\sum_{j}|\lz_j|\lf\{\int_{\widehat{B}_j}|a_j(x,t)|^2
\frac{\vz(B(x,t),\|\chi_{B_j}\|_\lv^{-1})}{t^n}
\frac{dx\,dt}{t}\r\}^{1/2}\\
&&\hs\times\lf\{\int_{\widehat{B}_j}|\phi_t\ast b(x)|^2
\frac{t^n}{\vz(B(x,t),\|\chi_{B_j}\|_\lv^{-1})}
\frac{dx\,dt}{t}\r\}^{1/2}\\
&&\ls\sum_{j}|\lz_j||B_j|^{-1/2}
\lf\{\int_{\widehat{B}_j}|a_j(x,t)|^2
\frac{dx\,dt}{t}\r\}^{1/2}
\|\chi_{B_j}\|_{L^\vz(\rn)}
\|d\mu\|_\vz\\
&&\ls\sum_{j=1}^\fz|\lz_j|\|d\mu\|_{\vz}
\ls\|f\|_{H^\vz(\rn)}\|d\mu\|_{\vz},
\end{eqnarray*}
which implies that $b\in\lys$ and $\|b\|_\lys\ls\|d\mu\|_{\vz}$.
This finishes the proof of Theorem \ref{carl}.
\end{proof}

\begin{rem}\label{r-carl}
(i) Fefferman and Stein \cite{fs72} shed some light on the tight
connection between BMO-function and Carleson measure,
which is the case of Theorem \ref{carl} when $s=0$ and
$\vz(x,t):=t$ for all $x\in\rn$ and $t\in(0,\fz)$.

(ii) When $s=0$, $\vz(x,t):=w(x)t$ and $w\in A_1(\rn)$,
Theorem \ref{carl} was obtained in \cite{hsv07}.

(iii) When $\vz$ is as in Remark \ref{r-def}(ii) with $p\in(0,1)$ and
$w\in A_1(\rn)$, Theorem \ref{carl} is new.
\end{rem}


\bigskip

Yiyu Liang and Dachun Yang (Corresponding author)

\medskip

School of Mathematical Sciences, Beijing Normal University,
Laboratory of Mathematics and Complex Systems, Ministry of
Education, Beijing 100875, People's Republic of China

\smallskip

{\it E-mails}: \texttt{yyliang@mail.bnu.edu.cn} (Y. Liang)

\hspace{1.55cm}\texttt{dcyang@bnu.edu.cn} (D. Yang)


\begin{thebibliography}{99}


%
%


\bibitem{ae02} A. A. Arkhipova and O. Erlhamahmy, Regularity of solutions to a
diffraction-type problem for nondiagonal linear elliptic systems
in the Campanato space. Function theory and applications, J. Math. Sci. (New York)
112 (2002), 3944-3966.

\vspace{-0.3cm}

\bibitem{bo31} Z. Birnbaum and W. Orlicz,
\"{U}ber die verallgemeinerung des begriffes der zueinander konjugierten potenzen,
Studia Math. 3 (1931), 1-67.

\vspace{-0.3cm}

\bibitem{bfg10} A. Bonami, J. Feuto, and S. Grellier,
Endpoint for the DIV-CURL lemma in Hardy spaces,
Publ. Mat. 54 (2010), 341-358.

\vspace{-0.3cm}

\bibitem{bg10} A. Bonami and S. Grellier,
Hankel operators and weak factorization for Hardy-Orlicz spaces,
Colloq. Math. 118 (2010), 107-132.

\vspace{-0.3cm}

\bibitem{bgk12} A. Bonami, S. Grellier and L. D. Ky,
Paraproducts and products of functions in $BMO(\mathbb R^n)$ and $H^1(\mathbb R^n)$ through wavelets,
J. Math. Pure Appl. 97 (2012), 230-241.

\vspace{-0.3cm}

\bibitem{bijz07} A. Bonami, T. Iwaniec, P. Jones and M. Zinsmeister,
On the product of functions in $BMO$ and $H^1$,
Ann. Inst. Fourier (Grenoble) 57 (2007), 1405-1439.

\vspace{-0.3cm}

\bibitem{c64} S. Campanato,
Propriet¨¤ di una famiglia di spazi funzionali,
Ann. Scuola Norm. Sup. Pisa (3) 18 (1964), 137-160.

%
%
\vspace{-0.3cm}

\bibitem{d05} L. Diening,
Maximal function on Musielak-Orlicz spaces and generalized Lebesgue spaces,
Bull. Sci. Math. 129 (2005), 657-700.

\vspace{-0.3cm}

\bibitem{dhr09} L. Diening, P. H\"ast\"o and S. Roudenko,
Function spaces of variable smoothness and integrability,
J. Funct. Anal. 256 (2009), 1731-1768.

\vspace{-0.3cm}

\bibitem{dxy07} X. T. Duong, J. Xiao and L. Yan,
Old and new Morrey spaces with heat kernel bounds,
J. Fourier Anal. Appl. 13 (2007), 87-111.

\vspace{-0.3cm}

\bibitem{ft11} L. Fattorusso and A. Tarsia,
Regularity in Campanato spaces for solutions of fully nonlinear elliptic systems, Discrete Contin. Dyn. Syst. 31 (2011), 1307-1323.

%
%

\vspace{-0.3cm}

\bibitem{fs72} C. Fefferman and E. M. Stein, $H^p$ spaces of several
variables, Acta Math. 129 (1972), 137-193.

\vspace{-0.3cm}

\bibitem{fs82} G. B. Folland and E. M. Stein,
Hardy spaces on homogeneous groups,
Princeton Univ. Press, Princeton, 1982.

\vspace{-0.3cm}

\bibitem{g79} J. Garc\'ia-Cuerva,
Weighted $H^p$ spaces,
Dissertationes Math. (Rozprawy Mat.) 162 (1979), 1-63.

\vspace{-0.3cm}

\bibitem{gr85} J. Garc\'{i}a-Cuerva and J. Rubio de Francia,
Weighted Norm Inequalities and Related Topics,
Amsterdam, North-Holland, 1985.

%

%

\vspace{-0.3cm}

\bibitem{g02} J. A. Griepentrog, Linear elliptic boundary value problems
with non-smooth data: Campanato spaces of functionals, Math. Nachr. 243 (2002), 19-42.

\vspace{-0.3cm}

\bibitem{hsv07} E. Harboure, O. Salinas and B. Viviani,
A look at $\bmo_\vz(w)$ through Carleson measures,
J. Fourier Anal. Appl. 13 (2007), 267-284.

\vspace{-0.3cm}

\bibitem{hyy} S. Hou, D. Yang and S. Yang,
Lusin area function and molecular characterizations of Musielak-Orlicz Hardy spaces and their applications,
arXiv: 1201.1945.

\vspace{-0.3cm}

\bibitem{hmy07} G. Hu, Y. Meng and D. Yang, Estimates for Marcinkiewicz
integrals in BMO and Campanato spaces, Glasg. Math. J. 49 (2007), 167-187.

%

%
%
%
\vspace{-0.3cm}

\bibitem{jn} F. John and L. Nirenberg, On functions of bounded mean
oscillation, Comm. Pure Appl. Math. 14 (1961), 415-426.

%
%

\vspace{-0.3cm}

\bibitem{j80} S. Janson,
Generalizations of Lipschitz spaces and an application to Hardy spaces and bounded mean oscillation,
Duke Math. J. 47 (1980), 959-982.

%
%
%
%
%
%
\vspace{-0.3cm}

\bibitem{ky} L. D. Ky,
New Hardy spaces of Musielak-Orlicz type and boundedness of sublinear operators,
arXiv: 1103.3757.

\vspace{-0.3cm}

\bibitem{ky2} L. D. Ky,
Bilinear decompositions and commutators of singular integral operators,
Trans. Amer. Math. Soc. (to appear) or arXiv: 1105.0486.

\vspace{-0.3cm}

\bibitem{l11} A. K. Lerner, Sharp weighted norm inequalities for Littlewood-Paley
operators and singular integrals, Adv. Math. 226 (2011), 3912-3926.

\vspace{-0.3cm}

\bibitem{l08} W. Li,
John-Nirenberg type inequalities for the Morrey-Campanato spaces,
J. Inequal. Appl. 2008, Art. ID 239414, 5 pp.

\vspace{-0.3cm}

\bibitem{lhy} Y. Liang, J. Huang and D. Yang,
New real-variable characterizations of Musielak-Orlicz Hardy spaces,
J. Math. Anal. Appl. 395 (2012), 413-428.

\vspace{-0.3cm}

\bibitem{ly} Y. Liang and D. Yang,
Intrinsic Littlewood-Paley function characterizations of Musielak-Orlicz Hardy spaces,
Submitted.

%


%
%

\vspace{-0.3cm}

\bibitem{mw76} B. Muckenhoupt and R. L. Wheeden,
Weighted bounded mean oscillation and the Hilbert transform,
Studia Math. 54 (1976), 221-237.

\vspace{-0.3cm}

\bibitem{m83} J. Musielak,
Orlicz Spaces and Modular Spaces,
Lecture Notes in Math. 1034, Springer-Verlag, Berlin, 1983.

\vspace{-0.3cm}

\bibitem{n10} E. Nakai,
Singular and fractional integral operators on Campanato spaces with variable growth conditions,
Rev. Mat. Complut. 23 (2010), 355-381.


\vspace{-0.3cm}

\bibitem{n07} E. Nakai,
The Campanato, Morrey and H\"older spaces on spaces of homogeneous type,
Studia Math. 176 (2006), 1-19.

\vspace{-0.3cm}

\bibitem{ny85} E. Nakai and K. Yabuta,
Pointwise multipliers for functions of bounded mean oscillation,
J. Math. Soc. Japan 37 (1985), 207-218.

\vspace{-0.3cm}

\bibitem{o32} W. Orlicz,
\"{U}ber eine gewisse Klasse von Raumen vom Typus B,
Bull. Int. Acad. Pol. Ser. A 8 (1932), 207-220.

%
%
%
%
%

\vspace{-0.3cm}

\bibitem{p69} J. Peetre, On the theory of ${\mathcal L}_{p,\lz}$ spaces,
J. Funct. Anal. 4 (1969), 71-87.

\vspace{-0.3cm}

\bibitem{s79} J.-O. Str\"{o}mberg,
Bounded mean oscillation with Orlicz norms and duality of Hardy spaces,
Indiana Univ. Math. J 28 (1979), 511-544.

\vspace{-0.3cm}

\bibitem{st89} J.-O. Str\"{o}mberg and A. Torchinsky,
Weighted Hardy Spaces,
Lecture Notes in Math. 1381, Springer-Verlag, Berlin, 1989.

%
%

\vspace{-0.3cm}

\bibitem{tw80} M. H. Taibleson and G. Weiss, The molecular characterization
of certain Hardy spaces. Representation theorems for Hardy spaces,
pp. 67-149, Ast\'erisque, 77, Soc. Math. France, Paris, 1980.

\vspace{-0.3cm}

\bibitem{v87} B. E. Viviani,
An atomic decomposition of the predual of $BMO(\rho)$,
Rev. Mat. Ibero. 3 (1987), 401-425.

%
%
%
%
%
%
%
%
%
%
%

%

%

\vspace{-0.3cm}

\bibitem{yy11} D. Yang and S. Yang,
New characterizations of weighted Morrey-Campanato spaces,
Taiwanese J. Math. 15 (2011), 141-163.

\vspace{-0.3cm}

\bibitem{ysy} W. Yuan, W. Sickel and D. Yang,
Morrey and Campanato Meet Besov, Lizorkin and Triebel,
Lecture Notes in Mathematics, 2005,
Springer-Verlag, Berlin, 2010, xi+281 pp.

\end{thebibliography}
\end{document}